\documentclass[a4paper,12pt,reqno]{amsart}
\newtheorem{theorem}{Theorem}
\newtheorem{lemma}{Lemma}
\newtheorem{proposition}{Proposition}
\newtheorem{corollary}{Corollary}
\newtheorem*{thmA}{Theorem A}
\theoremstyle{definition}

\theoremstyle{remark}

\def\beq{\begin{equation}}
\def\eeq{\end{equation}}
\numberwithin{equation}{section}
\newcommand{\R}{{\mathbb R}}

\newcommand{\E}{\mathbb E}
\newcommand{\B}{{\mathbb B}}
\newcommand{\Sp}{{\mathbb S}}

\newcommand{\calAp}{A_{\mathcal P}}

\newcommand{\calO}{{\mathcal O}}
\newcommand{\calM}{{\mathcal M}}
\newcommand{\calP}{{\mathcal P}}

\newcommand{\cp}{{\rm cap}\,}
\begin{document}
\title[]{On Lundh's percolation diffusion}
\author{Tom Carroll}
\address{Department of Mathematics\\
University College Cork\\
Cork, Ireland}
\email{t.carroll@ucc.ie}
\author{Julie O'Donovan}
\address{Department of Mathematics\\
University College Cork\\
Cork, Ireland}
\email{j.odonovan@ucc.ie}
\author{Joaquim Ortega-Cerd\`a }
\thanks{The third author is supported by the project MTM2008-05561-C02-01 and
 the grant 2009 SGR 1303}
\address{Departament de Matem\`atica Aplicada i An\`alisi\\
Universitat de Barcelona\\
Gran Via 585\\
08007 Barcelona, Spain.}
\email{jortega@ub.edu}

\date{\today}

\begin{abstract}
A collection of spherical obstacles in the ball in Euclidean space is said to be 
\textsl{avoidable\/} for Brownian motion if there is a positive probability that 
Brownian motion diffusing from some point in the ball will avoid 
all the obstacles and reach the boundary of the ball. 
The centres of the spherical obstacles are generated according to a Poisson point
process while the radius of an obstacle is a deterministic function depending only
on the distance from the obstacle's centre to the centre of the ball. 
Lundh has given the name \textsl{percolation diffusion\/} to this process if 
avoidable configurations are generated with positive probability. 
An integral condition for percolation diffusion is derived in terms of  the 
intensity of the Poisson point process and the function that determines the radii of the 
obstacles. 
\end{abstract}

\maketitle

\section{Introduction}
Lundh proposed in \cite{Lundh} a percolation model in the ball 
$\B = \{x\in \R^d: \vert x \vert< 1\}$, $d\geq 3$, involving diffusion through a random 
collection of spherical obstacles. In Lundh's formulation, the radius of an 
obstacle is proportional to the distance from its centre to the boundary 
$\Sp = \{x\in \R^d: \vert x \vert = 1\}$ of the ball. 
The centres of the obstacles are generated at random by a Poisson point process with
a spherically symmetric intensity $\mu$. 
Lundh called a random collection of obstacles is \textsl{avoidable\/} if 
Brownian motion diffusing from a point in the ball $\B$ 
has a positive probability of reaching the outer boundary $\Sp$ 
without first hitting any of the obstacles. 
Lundh set himself the task of characterising those Poisson intensities $\mu$ which 
would generate an avoidable collection of obstacles with positive probability, 
and named this phenomenon \textsl{percolation diffusion\/}.
Our main objective herein is to extend Lundh's work by removing some of his 
assumptions on the Poisson intensities and on the radii of the obstacles. 

Deterministic configurations of obstacles in two dimensions are considered in detail 
by Akeroyd \cite{Akeroyd} and by Ortega-Cerd\`a and Seip \cite{OS}, 
while O'Donovan \cite{Julie} and Gardiner and Ghergu \cite{GG} consider 
configurations in higher dimensions. 
The result below is taken from these articles. 
First some notation is needed. 
Let $B(x,r)$ and $S(x,r)$ stand for the Euclidean ball and sphere, respectively, with centre $x$ 
and radius $r$ and let $\overline{B}(x,r)$ stand for the closed ball with this centre and radius. 
Let $\Lambda$ be a countable set of points in the ball $\B$ which is 
\textsl{regularly spaced\/} in that it has the following properties
\begin{enumerate}
\item[(a)] there is a positive $\epsilon$ such that if $\lambda$, $\lambda' \in \Lambda$, $\lambda \ne \lambda'$ and $\vert \lambda \vert \geq \vert \lambda'\vert$ then 
\beq\label{1}
\vert \lambda-\lambda'\vert \geq \epsilon\big( 1-\vert \lambda\vert \big).
\eeq
\item[(b)] there is an $r <1$ such that 
\beq\label{2}
\B = \bigcup_{\lambda \in \Lambda} B\big(\lambda,r(1-\vert\lambda\vert)\big).
\eeq
\end{enumerate} 
Let $\phi:[0,1) \to [0,1)$ be a decreasing function such that the closed balls 
$\big\{ \overline{B} \big(\lambda,\phi(\vert \lambda \vert \big) \big\}$, 
$\lambda \in \Lambda$, are disjoint, and set 
\[
\calO =  \bigcup_{\lambda\in\Lambda}
\overline{ B }\big(\lambda,\phi(\vert \lambda \vert \big) .
\]
Avoidability of the collection of spherical obstacles $\calO$ 
is equivalent to the harmonic measure condition $\omega(x,\Sp,\Omega) >0$, 
where $ \Omega = \B \setminus \calO$ 
and $x$ is some (any) point in the domain $\Omega$.
\begin{thmA}\label{thmA} 
The collection of spherical obstacles $\calO$ is avoidable if and only if 
\beq\label{4}
\int_0^1 \frac{dt}{(1-t)\log\big( (1-t)/\phi(t) \big)} < \infty \quad\mbox { if }\ d=2,
\eeq
\beq\label{5}
\int_0^1 \frac{\phi(t)^{d-2}}{(1-t)^{d-1}} < \infty \quad\mbox { if }\ d\geq3.
\eeq
\end{thmA}
Our goal is is to obtain a counterpart of this result for a random 
configuration of obstacles. 
We work with a Poisson random point process on the Borel subsets of the ball $\B$
with mean measure  $d\mu(x) = \nu(x)\,dx$ which is absolutely continuous 
relative to Lebesgue measure.
(It\^o presents a complete, concise treatment of this topic in Section~1.9 of 
his book \cite{Ito}). 
The radius function $\phi$ and the intensity function $\nu$  are assumed 
to satisfy, for some $C>1$, 
\beq\label{phinu}
\begin{cases}\frac{1}{C} \phi(x) \leq \phi(y) \leq C \phi(x)\\  
\frac{1}{C}\nu(x) \leq \nu(y) \leq C \nu(x)
\end{cases}\ \mbox{ if}\   y \in B\left(x,\frac{1-\vert x\vert}{2}\right).
\eeq
 It is also assumed that 
 \beq\label{phi}
 \frac{\phi(\vert x\vert )}{1-\vert x\vert} \leq c < 1 \ \mbox{ for }\ x\in \B.
 \eeq
and that 
\beq\label{nu}
(1-\vert x \vert) \phi(x)^{d-2} \nu(x) = O\left( \frac{1}{1-\vert x\vert} \right) 
\mbox{ as } \vert x \vert \to 1^-. 
\eeq
Let $\mathcal P$ be a realisation of points from this Poisson random point process 
and let 
\beq\label{6}
A_{\mathcal P} = \bigcup_{p \in \mathcal P} 
\overline{ B }\big(p,\phi(p) \big) , 
\qquad \Omega_{\mathcal P} = \B \setminus A_{\mathcal P},
\eeq
so that $\Omega_{\mathcal P}$ is an open, though not necessarily connected, subset 
of $\B$.  The archipelago of spherical obstacles $A_\calP$ is said to be \textsl{avoidable\/} 
if there is a positive probability that Brownian motion diffusing from some point in 
$\Omega_\calP$ reaches the unit sphere $\Sp$ before hitting the obstacles $A_\calP$, 
that is if the harmonic measure of $A_\calP$ relative to $\Omega_\calP$ satisfies
$\omega(x, A_\calP,\Omega_\calP)< 1$ for some $x$ in $\Omega_\calP$. 
If $\Omega_\calP$ is connected then, by the maximum principle, this condition 
does not depend on $x \in \Omega_\calP$.
We do not insist, however, on the configuration being avoidable for Brownian motion 
diffusing from the origin. 

We have \textsl{percolation diffusion\/} if there is a positive probability that the 
realisation of points from the Poisson random point process results in an 
avoidable configuration.
Our main result is 
\begin{theorem}\label{thm1}Suppose that \eqref{phinu}, 
\eqref{phi} and \eqref{nu} hold. 
Percolation diffusion occurs if and only if there is a set of points $\tau$ 
of positive measure on the sphere such that 
\beq\label{7*}
\int_\B \frac{(1-\vert x\vert^2)^2}{\vert x-\tau\vert^d} 
\,\phi(x)^{d-2} \nu(x)\,dx < \infty.
\eeq
Thus the random archipelago $\calAp$ is avoidable with positive probability 
if and only if the Poisson balayage of the measure 
$(1-\vert x \vert^2)\phi(x)^{d-2} \nu(x)\,dx$
is bounded on a set of positive measure on the boundary of the unit ball.

Furthermore,  in the case of percolation diffusion the random archi-pelago 
$\calAp$ is avoidable with probability one. 
\end{theorem}
In the radial case the following corollary follows directly from Theorem~\ref{thm1}. 
\begin{corollary}\label{cor1}Suppose that, in addition to \eqref{phinu}, 
\eqref{phi} and \eqref{nu}, the intensity $\nu$ and the 
radius function $\phi$ are radial in that they depend only on $\vert x\vert$. Then 
percolation diffusion occurs if and only if 
\beq\label{7}
\int_0^1 (1-t)\, \phi(t)^{d-2}\, \nu(t) \,dt < \infty.
\eeq
\end{corollary}
Lundh's result \cite[Theorem~3.1]{Lundh} is the case $\phi(t) = c(1-t)$ of this corollary, 
in which case \eqref{7} becomes
\beq\label{8}
\int_0^1 (1-t)^{d-1} \nu(t)\,dt < \infty.
\eeq
This corresponds to the condition stated by Lundh that the radial intensity function 
should be integrable on $(0,\infty) $ when allowance is made for the fact that 
he works in the hyperbolic unit ball.  
As pointed out in \cite{JuliePhD}, Lundh's deduction from \eqref{8} 
(see \cite[Remark~3.2]{Lundh}) that percolation diffusion can only occur when 
the expected number of obstacles in a configuration is finite isn't correct. 
In fact, \eqref{8} holds in the case $\nu(t) = (1-t)^{1-d}$ and we have percolation diffusion. 
At the same time, the expected number of obstacles $N(\B)$ in the ball is 
\beq\label{9}
\E[N(\B)] = \int_\B d\mu(x) = \int_\B \frac{dx}{(1-\vert x \vert)^{d-1}} = \infty.
\eeq
Lundh's remark erroneously undervalues his work since it gives the 
impression that, in his original setting, percolation diffusion can only occur 
if the number of obstacles in a configuration is finite almost surely.   

The intensity $\nu(t) = 1/(1-t)^d$ corresponds, in principle, to a regularly spaced
collection of points since the expected number of points in a Whitney cube $Q$ 
of sidelength $\ell(Q)$ and centre $c(Q)$ is, in the case of this intensity,
\[
\E[N(Q)] = \int_Q d\mu(x) \sim \nu\big( c(Q) \big) {\rm Vol}(Q) 
= \frac{\ell(Q)^d}{(1-\vert c(Q)\vert)^d} \sim {\rm constant}.
\]
We note that there is agreement in principle between the integral condition \eqref{5} 
for the deterministic setting and the integral condition \eqref{7} with  
$\nu(t) = 1/(1-t)^d$ for the random setting.
\section{Avoidability, minimal thinness and a Wiener-type criterion}
\label{Sec2} 
Avoidability of a realised configuration of obstacles $A_\calP$ may be reinterpreted 
in terms of minimal thinness of $A_\calP$ at points on the boundary of the unit ball
(see \cite{Armitage} for a thorough account of minimal thinness).  
This is Lundh's original approach, 
and is also the approach adopted by the authors of \cite{Julie,JuliePhD,GG}.  

For a positive superharmonic function $u$ on $\B$ and a closed subset $A$
of $\B$, the reduced function $R_u^A$ is defined by
\[
R_u^A = \inf \big\{ v \colon v \text{ is positive and superharmonic on $\B$ 
and $v \geq u$ on }A \big\}.
\]
The set $A$ is minimally thin at $\tau\in\Sp$ if there is an $x$ in $\B$ at which 
the reduced function of the Poisson kernel $P(\cdot,\tau)$ for $\B$ with pole at $\tau$ 
satisfies $R_{P(\cdot,\tau)}^A(x) < P(x,\tau)$.
Minimal thinness in this context has been characterised in terms
of capacity by Ess\'en \cite{Essen} in dimension 2 and by Aikawa \cite{Aikawa} in higher dimensions. 
Let $\{Q_k\}_{k=1}^\infty$ be a Whitney decomposition of the ball $\B$ into cubes so that, in particular, 
\[
{\rm diam}(Q_k) \leq \text{dist}(Q_k,\Sp) \leq 4\, {\rm diam}(Q_k).
\]
Let $\ell(Q_k)$ be the sidelength of $Q_k$. 
Let $\cp(E)$ denote the Newtonian capacity of a Borel set $E$.  
Aikawa's criterion for minimal thinness of $A$ at a boundary point 
$\tau$ of  $\B$ is that the series $W(A,\tau)$ is convergent, where 
\beq\label{2.1}
W(A,\tau) = 
\sum_k \frac{\ell(Q_k)^2}{\rho_k(\tau)^d}\, \cp(A \cap Q_k),
\eeq
$\rho_k$ being the distance from $Q_k$ to the boundary point $\tau$.  
A proof of the following proposition can be found in \cite[Page 323]{GG}. 
The proof goes through with only very minor modifications even though
we do not insist on evaluating harmonic measure at the origin and 
the open set $\B\setminus A$ may not be connected. 
\begin{lemma}\label{Prop.mt}
Let $A$ be a closed subset of $\B$. Let
\beq\label{2.2}
\calM =\{\tau \in \Sp : A \text{ is minimally thin at }\tau\}.
\eeq
Then $A$ is avoidable if and only if $\calM$ has positive measure on $\Sp$, that is 
if and only if $W(A,\tau)<\infty$ for a set of $\tau$ of positive measure on $\Sp$.
\end{lemma}
The question of whether a given set $A$ is avoidable for Brownian motion
is thereby reduced to an estimation of capacity. 

The following zero-one law simplifies the subsequent analysis, and will 
imply that the random archipelago is avoidable with probability zero or 
probability one, as stated in Theorem~1.
Again, $\tau$ is used to denote points on the sphere $\Sp$ and 
$A_\calP$ denotes an archipelago constructed as in \eqref{6} from 
a random realisation $\calP$ of points taken from the Poisson point process. 
\begin{lemma}\label{Lemma1}
The event that $A_\calP$ is minimally thin at $\tau$ has probability 0 or 1.  
\end{lemma}
\begin{proof}
Whether or not the set $A_\calP$ is minimally thin at $\tau$ depends on 
the convergence of the series $W(A_\calP,\tau)$. 
Partition the cubes $\{Q_k\}_1^\infty$ into finitely many disjoint groups 
$\{Q^i_k\}_{k=1}^\infty$, $i=1$, 2, $\ldots$, $n$, so that any ball 
in $A_\calP$ can meet at most one cube in each group. 
Then break the summation $W(A_\calP,\tau)$ into corresponding summations 
\beq\label{Xk}
W^i(A_\calP,\tau) = \sum_{k=1}^\infty X^i_k \quad \mbox{where }
X_k = \frac{\ell(Q_k)^2}{\rho_k(\tau)^d}\, \cp(A \cap Q_k)
\eeq
The random variables $X^i_k$ in each resulting summation are independent. 
The event $W^i(A_\calP,\tau)< \infty$ belongs to the tail field of the 
corresponding $X^i_k$'s, hence this event has probability 0 or 1. 
It follows that the event $W(A_\calP,\tau)< \infty$ has probability 0 or 1. 
\end{proof}
\section{The expected value of the Wiener-type criterion and the Poisson balayage}
\label{Sec3} 
The proof of Theorem~\ref{thm1} follows the outline of Lundh's argument \cite{Lundh} 
and the second author's thesis \cite{JuliePhD}. 

We work with a Poisson point process in the ball. Each realisation $\calP$ of this process 
gives rise to an archipelago $\calAp$ via \eqref{6}, which is avoidable for 
Brownian motion if and only if the associated Wiener-type series $W(\calAp,\tau)$ 
is finite for a set of $\tau$ of positive measure on the sphere $\Sp$. 
For a fixed $\tau$ on the sphere $\Sp$, the series $W(\calAp,\tau)$ is a random variable. 
Proposition~\ref{prop1} states that its expected value is comparable to the Poisson 
balayage \eqref{7*}. 
We denote by $c$ and $C$ any positive finite numbers whose values depend 
only on dimension and are immaterial to the main argument. 
%
%
\begin{proposition}\label{prop1} 
Fix a point $\tau$ on the sphere $\Sp$. Then 
\beq\label{EsimBal}
\E\big[W(\calAp,\tau)\big]
\sim
\int_\B \frac{(1-\vert x\vert^2)^2}{\vert \tau - x \vert^d} \, 
	\phi(x)^{d-2}\,\nu(x)\,dx.
\eeq
\end{proposition}
The proof of Proposition~\ref{prop1} depends on a two-sided estimate for 
the expected value of the capacity of the intersection of a Whitney cube 
$Q_k$ with the set of obstacles $\calAp$  in terms of the mean measure 
$\mu(Q_k)$ of the cube and a typical value of the radius function $\phi$ on the cube.
%
%
\begin{lemma}\label{Lemma3}
For a Whitney cube $Q$ and any point $x \in Q$, 
\beq\label{4.2}
\E[\cp (\calAp\cap Q)]  \sim  \phi(x)^{d-2}\, \mu(Q).
\eeq
\end{lemma}
Lundh did not require an estimate of this type as the size of one of his obstacles 
was comparable to the size of the Whitney cube containing its centre. 
The capacity of $\calAp\cap Q$ therefore depended only on the 
probability of whether of not the cube $Q$ contained a point 
from the Poisson point process. 
We first deduce Proposition~\ref{prop1} from Lemma~\ref{Lemma3} and 
then prove Lemma~\ref{Lemma3}.

\begin{proof}[Proof of Proposition~\ref{prop1}]
The upper bound for $\E[\cp (\calAp\cap Q)]$ in Lemma~\ref{Lemma3} 
leads to an upper bound for the expected value of Aikawa's series 
\eqref{2.1} with $A = A_\calP$ as follows:
\begin{align*}
\E[W(A_\calP,\tau)]& = \E\left[\sum_k \frac{\ell(Q_k)^2}{\rho_k(\tau)^d}
\,\cp(A_\calP\cap Q_k)\right] \\
&=
\sum_k\frac{\ell(Q_k)^2}{\rho_k(\tau)^d} \,\E[\cp (A_\calP\cap Q_k)] \\ 
&\leq 
C\, \sum_k   \frac{\ell(Q_k)^2}{\rho_k(\tau)^d}\,  \phi( x_k )^{d-2}\, \mu(Q_k) 
\end{align*}
where $x_k$ is any point in $Q_k$. 
Since the radius function $\phi$ is approximately constant on each Whitney cube by 
\eqref{phinu}, it follows that
\begin{align*}
\E[W(A_\calP,\tau)] &\leq 
C\, \sum_k \int_{Q_k} \frac{(1-\vert x\vert^2)^2}{\vert \tau - x \vert^d} \, 
	\phi(x)^{d-2}\,\nu(x)\,dx\\
& = 
C \int_\B \frac{(1-\vert x\vert^2)^2}{\vert \tau - x \vert^d} \, 
	\phi(x)^{d-2}\,\nu(x)\,dx.
\end{align*}

In the other direction, first choose a point $x_k$ in each Whitney cube $Q_k$.
Then, 
\begin{align*}
\int_\B \frac{(1-\vert x\vert^2)^2}{\vert \tau - x \vert^d}\, 
	\phi(x)^{d-2}\,\nu(x)\,dx
& \leq C 
\sum_k  \frac{\ell(Q_k)^2}{\rho_k(\tau)^d}\,  \phi( x_k )^{d-2}\, \mu(Q_k) \\
&\leq C
\sum_k\frac{\ell(Q_k)^2}{\rho_k(\tau)^d} \,\E[\cp (A_\calP\cap Q_k)] \\
& = C\, \E[W(A_\calP,\tau)],
\end{align*}
where the second inequality comes from the lower bound for 
$\E[\cp (A_\calP\cap Q_k)]$ in Lemma~\ref{Lemma3}.
\end{proof}
\begin{proof}[Proof of Lemma~\ref{Lemma3}]

The assumption \eqref{phi} implies that if an obstacle meets a Whitney cube $Q$ 
then its centre can lie in at most some fixed number $N$ of Whitney cubes 
neighbouring the cube $Q$. 
We label these cubes $Q^i$, where the index $i$ varies from 1 to at most $N$, 
and write $Q'$ for their union.  
Both the distance to the boundary, and the distance to a specific boundary point, 
are comparable in $Q$ and in $Q'$. 
Analogously, an obstacle with a centre in a specified cube can intersect 
at most some fixed number of neighbouring cubes. 

Consider a random realisation of points $\calP$ and a Whitney cube $Q_k$. 
By \eqref{phinu} the radius function $\phi$ is roughly constant on the cubes $Q_k^i$, 
say $\phi(x) \sim \phi( x_k)$, $x\in Q_k'$, where $x_k$ is any point chosen in $Q_k$.
Therefore, by the subadditivity property of capacity, 
\[
\cp(A_{\calP}\cap Q_k) \leq C\,  \phi(x_k)^{d-2}\,N(Q'_k),
\]
where $N(Q'_k)$ is the number of centres from the realised point process 
$\calP$ that lie in the union of cubes $Q'_k$. Taking the expectation leads to
\[ 
\E[\cp (\calAp\cap Q_k)]  \leq C\,  \phi(x_k )^{d-2} \E\left[N(Q'_k)\right]
 = C \,  \phi(x_k )^{d-2} \mu(Q'_k).
\]
By \eqref{phinu}, $\mu(Q'_k) \leq C \mu(Q_k)$ and the upper bound for 
$\E[\cp (\calAp\cap Q_k)]$ in Lemma~\ref{Lemma3} follows. 

In the other direction we proceed, as did Gardiner and Ghergu \cite{GG}, 
by employing the following super-additivity property of capacity 
due to Aikawa and Borichev \cite{AB}. 
Let $\sigma_d $ be the volume of the unit ball.
Let $F = \bigcup B(y_k,\rho_k)$ be a union of balls which lie inside some ball 
of unit radius. 
Suppose also that $\rho_k \leq 1/\sqrt{\sigma_d 2^d}$ for each $k$ 
and that the larger balls $B(y_k, \sigma_d^{-1/d}\rho_k^{1-2/d})$ are disjoint. 
Then
\beq\label{3.3}
\cp(F) \geq c\sum_k \cp \big( B(y_k,\rho_k) \big) = c \sum_k \rho_k^{d-2}. 
\eeq

Let $\phi_0$ be the minimum of $\phi(x)$ for $x$ in $Q$. 
By \eqref{phinu}, $\phi_0$ is comparable to $\phi(x)$ for any $x$ in $Q$. 
We only consider obstacles with centres in $Q$ and suppose that all such obstacles  
have radius $\phi_0$, since in so doing the capacity of $\calAp\cap Q$ decreases.
Set 
\[
\alpha = \min\big\{ \big( \ell(Q)\sqrt{d}\big)^{-1}, 
\big( \sqrt{\sigma_d\,2^d}\, \phi_0\big)^{-1} \big\} 
\]
and set 
\[
N  = \lfloor 4^{-1}\, \ell(Q)\, \sigma_d^{1/d}\, 
\alpha^{2/d}\, \phi_0^{2/d-1}\rfloor. 
\]
By \eqref{phi}, we have $\alpha \geq c/\ell(Q)$.
The cube $Q$ is divided into $N^d$ smaller cubes each of sidelength $\ell(Q)/N$:
we write $Q'$ for a typical sub-cube. 
Inside each cube $Q'$ consider a smaller concentric cube $Q''$ of 
sidelength $\ell(Q)/(4N)$. 
If a cube $Q''$ happens to contain points from the realisation $\calP$ 
of the random point process, we choose one such point.
This results in points $\lambda_1$, $\lambda_2$, $\ldots$, $\lambda_m$, 
say, where $m \leq N^d$. 
By the choice of $\alpha$ and $N$, each ball $B(\lambda_k, \phi_0)$ 
is contained within the sub-cube $Q'$ that contains its centre. 
We set 
\[
A_{\mathcal{P}, Q} = \bigcup_{k=1}^m \overline{B}(\lambda_k, \phi_0). 
\]
Since $A_{\mathcal{P}, Q} \subset \calAp\cap Q$, if follows from 
monotonicity of capacity that 
$\cp\big( \calAp\cap Q \big) \geq \cp\big( A_{\mathcal{P}, Q}\big)$. 

To estimate the capacity of $A_{\mathcal{P}, Q}$, we scale the cube $Q$ by 
$\alpha$. 
By the choice of $\alpha$, the cube $\alpha Q$ lies inside a ball of unit radius
and the radius of each scaled ball from $A_{\mathcal{P}, Q}$ satisfies
$\alpha \phi_0 \leq (\sigma_d 2^d)^{-1/2}$. 
The only condition that remains to be checked before applying Borichev and Aikawa's
estimate \eqref{3.3} to the union of balls $\alpha A_{\mathcal{P}, Q}$ is that 
the balls with centre $\alpha \lambda_k$ and radius 
$\sigma_d^{-1/d}(\alpha \phi_0)^{1-2/d}$ are disjoint. 
They are if 
\[
2 \sigma_d^{-1/d}(\alpha\,\phi_0)^{1-2/d} \leq \frac{\alpha\, \ell(Q)}{2N}
\]
since the centres of the balls are at least a distance $\alpha\, \ell(Q)/(2N)$ apart. 
This inequality follows from the choice of $N$. 
Applying \eqref{3.3} and the scaling law for capacity yields 
\[
\cp\big( A_{\mathcal{P}, Q}\big) 
= \alpha^{2-d}\cp\big(\alpha A_{\mathcal{P}, Q}\big) 
\geq 
\alpha^{2-d}\, c\, X\, (\alpha\phi_0)^{d-2} = 
c\,X\,\phi_0^{d-2},
\]
where $X= m$ is the number of sub-cubes $Q''$ of $Q$ in our construction 
that contain at least one point of $\mathcal{P}$. Hence, 
\beq\label{3.4}
\E[\cp (\calAp\cap Q)]  \geq c\, \phi_0^{d-2}\, \E[X].
\eeq
The probability that a particular sub-cube $Q''$ contains a point of $\mathcal{P}$ 
is 
\[
1-\mathbb{P}\big( \mathcal{P} \cap Q'' = \emptyset\big) 
= 1-e^{ - \mu(Q'')},
\]
by the Poisson nature of the random point process. 
For any sub-cube $Q''$ with centre $x$, say,
\begin{align*}
\mu(Q'') & \sim \nu(x) \left( \frac{\ell(Q)}{N} \right)^d
 \sim \nu(x) \frac{\phi_0^{d-2}}{\alpha^2}& \mbox{(by choice of $N$)}\\
&  \leq \nu(x)\, \ell(Q)^2\, \phi_0^{d-2} &  \mbox{(since $\alpha \geq c/l(Q)$)}\\
& = O(1) & \mbox{(by \eqref{nu})}. 
\end{align*} 
It then follows that 
\beq\label{3.5}
\E [X] = \sum_{Q''\subset Q} 1-e^{ - \mu(Q'')} \geq \sum_{Q''\subset Q} \mu(Q'') \geq c \mu(Q),
\eeq
the last inequality being a consequence of the assumption \eqref{phinu}
and the fact that the volume of the union of the cubes $Q''$ is some fixed 
fraction of the volume of $Q$.
When combined with \eqref{3.4}, the estimate \eqref{3.5} yields the lower bound for 
$\E[\cp (\calAp\cap Q)]$. 
\end {proof}
%
%
\section{Proof of Theorem~\ref{thm1}} 
\label{Sec4}
To begin with, we need the following result from Lundh's paper \cite{Lundh}.
\begin{lemma}\label{Lemma4.1} 
Let $\tau\in \Sp$. Then $\E[W(\calAp,\tau)]$ is finite if and only if 
the series $W(\calAp,\tau)$ is convergent for almost all random configurations $\calP$.
\end{lemma}
\begin{proof}
It is clear that $\E[W(\calAp,\tau)]$ being finite implies that $W(\calAp,\tau)$ 
is almost surely convergent. 
The reverse direction is proved by Lundh \cite[p.\ 241]{Lundh} using 
Kolmogorov's three series theorem. 
Indeed, it is a consequence of this result \cite[p.\ 118]{Chung} that, 
in the case of a uniformly bounded sequence of non-negative independent 
random variables, the series $\sum_k X_k$ converges almost surely 
if and only if $\sum_k \E [X_k]$ is finite.
As in the proof of Lemma~\ref{Lemma1}, the series $W(\calAp,\tau)$ 
is split into $n$ series $W^i(\calAp,\tau) = \sum_{k=1}^\infty X^i_k$, 
each of which is almost surely convergent by assumption. 
The random variables $X_k$ in \eqref{Xk} are uniformly bounded.  
It then follows that $\sum_k \E [X_k^i] = \E \big[ \sum_k X^i_k\big]$ is convergent,
that is $\E[W^i(\calAp,\tau)]$ is finite. Summing over $i$, we find that 
$\E[W(\calAp,\tau)]$ is finite as claimed.
\end{proof}
\begin{proof}[Proof of Theorem~\ref{thm1}]
Let us first assume that the finite Poisson balayage condition \eqref{7*} holds 
for all $\tau$ in a set $T$, say, of positive measure $\sigma(T)$ 
on the boundary of the unit ball
and deduce from this that percolation diffusion occurs. 
In fact, we will show more -- we will show that the random archipelago is 
avoidable with probability one. 
By Proposition~\ref{prop1}, 
we see that $\E[W(\calAp,\tau)]$ is finite for $\tau \in T$, hence 
the series $W(\calAp,\tau)$ is convergent a.s.\ for each $\tau \in T$. 
For $\tau \in T$, set
\[
F_\tau = \{ \calP \colon W(\calAp,\tau) < \infty\}.
\]
so that $F_\tau$ has probability 1. We have 
\[
1 = \frac{1}{\sigma(T)}\int_T\E[ 1_{F_\tau}]\, d\tau = 
\E \left[\frac{1}{\sigma(T)}\int_T 1_{F_\tau}\, d\tau \right],
\]
from which it follows that $\int_\Sp 1_{F_\tau}\, d\tau = \sigma(T)$ 
with probability one. 
Equivalently, it is almost surely true that $\calP \in F_\tau$ for a.e.\ $\tau \in T$.
In other words, it is almost surely true that $\calAp$ is minimally thin at a set of 
$\tau$ of positive measure on the sphere $\Sp$, hence $\calAp$ is 
almost surely avoidable by Lemma~\ref{Prop.mt}.

Next we prove the reverse implication. 
For a random configuration $\calP$, set 
\[
M_\calP = \{ \tau \in \Sp : \calAp \mbox{ is minimally thin at } \tau\},
\]
similar to \eqref{2.2}. 
Suppose that percolation diffusion occurs. 
Then, with positive probability, $\calAp$ is minimally thin at each point 
of some set of positive surface measure on the sphere, so that
$\E\left[ \int_\Sp 1_{M_\calP}(\tau)\,d\tau\right] >0$. 
Interchanging the order of integration and expectation, we conclude that 
there is a set $T$ of positive measure on the sphere $\Sp$ such that 
$\mathbb{P}\big( \tau \in M_\calP \big) >0$ for $\tau \in T$. 
By Lemma~\ref{Lemma1}, $W(\calAp,\tau) < \infty$ a.s.\ for $\tau \in T$. 
By Lemma~\ref{Lemma4.1}, $\E[W(\calAp,\tau)]$ is finite for $\tau \in T$.
Finally, it follows from Proposition~\ref{prop1} that, 
for $\tau$  in the set $T$ of positive measure on the sphere $\Sp$, 
the Poisson balayage \eqref{7*} is finite. 
\end{proof}

\begin{proof}[Proof of Corollary~\ref{cor1}]In the case that both $\phi$ and 
$\nu$ are radial, the value of the Poisson balayage in \eqref{7*} is independent 
of $\tau \in \Sp$ and equals 
\[
\int_0^1 (1-t^2) \phi(t)^{d-2} \left( \int_{t\Sp}
	\frac{1-\vert x \vert^2}{\vert \tau - x\vert^d}\,d\sigma(x)\right)\nu(t)\,dt,	
\]
where $d\sigma$ is surface measure on the sphere $t\Sp$. 
Hence \eqref{7*} is equivalent to \eqref{7} in the radial setting. 
\end{proof}
%
%
\section{Percolation diffusion in space}
\label{sec5}
The Wiener criterion for minimal thinness of a set $A$ at $\infty$ in $\R^d$,
$d\geq 3$, is 
\beq\label{5.1}
W(A,\infty) = \sum_k \frac{\cp(A\cap Q_k)}{\ell(Q_k)^{d-2}} < \infty
\eeq
(see \cite{GG}, for example) where the cubes $\{Q_k\}$ are obtained by partitioning the 
cube of sidelength $3^j$ (centre 0 and sides parallel to the coordinate axes) into 
$3^{jd}$ cubes of sidelength $3^{j-1}$ and then deleting the central cube. 
Assuming that the radius function $\phi$ and the intensity of the Poisson process
$\nu$ are roughly constant on each cube $Q_k$ (+ version of \eqref{nu}),
the relevant version of Lemma~\ref{Lemma3} is that, for a cube $Q_k$ 
and any point $x_k \in Q$, 
\beq\label{5.2}
\E[\cp (\calAp\cap Q_k)]  \sim  \phi(x_k)^{d-2}\, \mu(Q_k), 
\eeq
and the relevant version of Proposition~\ref{prop1} is 
\beq\label{5.3}
\E\big[W(\calAp,\tau)\big]
\sim
\int_{\R^d\setminus\B} \left(\frac{\phi(x)}{\vert x \vert}\right)^{d-2} \, 
	\nu(x)\,dx.
\eeq
Since the random archipelago $\calAp$ is avoidable in this setting precisely
when it is minimally thin at the point at infinity, the criterion for percolation 
diffusion is that the integral on the right hand side of \eqref{5.3} be finite. 
Again this agrees in principle with a criterion for avoidability in the deterministic,
regularly located setting \cite[Theorem~2]{CO} (see also \cite[Theorem~6]{GG})
which corresponds to constant $\nu$ and $\phi$ radial, namely
\[
\int_1^\infty r \phi(r)^{d-2}\,dr <\infty.
\]

\begin{thebibliography}{99}
%
%
\bibitem{Aikawa}Aikawa, H.: \textsl{Thin sets at the boundary,} 
Proc. London Math. Soc. (3) {\bf 65} (1992), 357--382.
%
\bibitem{AB}Aikawa, H. and A. A. Borichev,
\textsl{Quasiadditivity and measure property of capacity and the tangential boundary behavior of 
harmonic functions,\/} Trans. Amer. Math. Soc. {\bf 348} (1996) 1013--1030.
%
\bibitem{Akeroyd}Akeroyd, J.R.: \textsl{Champagne subregions of the disk 
whose bubbles carry harmonic measure,\/} Math. Ann. {\bf 323} (2002) 267--279.
%
\bibitem{Armitage}Armitage, D.H. and S.J. Gardiner, 
\textsl{Classical potential theory}. Springer Monographs in Mathematics. 
Springer-Verlag London, London, 2001.
%
\bibitem{CO}Carroll, T. and J.\ Ortega-Cerd\`a: \textsl{Configurations of balls 
in Euclidean space that Brownian motion cannot avoid,\/} 
Ann.\ Acad.\ Sci.\ Fenn.\ Math.\ {\bf 32} (2007), 223--234.
%
\bibitem{Chung}Chung, K. L.: \textsl{A course in Probability Theory,\/} Academic Press, New York (1974). 
%
\bibitem{Essen}Ess\'en, M: \textsl{On minimal thinness, reduced functions and Green potentials,} Proc. Edinburgh Math. Soc. (Series 2) {\bf 36} (1993), 87--106. 
%
\bibitem{GG}Gardiner, S.J. and M.\ Ghergu: \textsl{Champagne subregions of the unit 
ball with unavoidable bubbles,\/} Ann.\ Acad.\ Sci.\ Fenn.\ Math.\ {\bf 35} (2010), 
321--329.
%
\bibitem{Ito}Ito, K.: \textsl{Stochastic Processes,\/} Springer-Verlag Berlin 
Heidelberg (2004).
%
\bibitem{Lundh}Lundh, T.: \textsl{Percolation Diffusion,\/} Stochastic Process. Appl. 
{\bf 95} (2001), 235--244.
%
\bibitem{Julie}O'Donovan, J.: {\sl Brownian motion in a ball in the presence of 
spherical obstacles,\/} Proc.\ Amer.\ Math.\ Soc.\ {\bf 138} (2010), 1711--1720. 
%
\bibitem{JuliePhD}O'Donovan, J.: {\sl Brownian motion in the presence of 
spherical obstacles,\/} PhD Thesis, University College Cork, Ireland, 2010. 
%
\bibitem{OS}Ortega-Cerd\`a, J. and K.\  Seip, {\sl Harmonic measure and 
uniform densities,} Indiana Univ.\ Math.\ J.\ {\bf 53} (2004), 905--923.
%
%
\end{thebibliography}
\end{document}